\documentclass[oneside,11pt]{amsart}
\usepackage{amsmath}
\usepackage{amsthm}
\usepackage{pinlabel}
\usepackage{epstopdf}
\usepackage{mathtools} 

\usepackage[T1]{fontenc}

\usepackage[small]{caption}

\usepackage{amsfonts}
\usepackage{amssymb}
\usepackage{amsxtra}

\newtheorem{thm}{Theorem}[section]

\newtheorem{prop}[thm]{Proposition}
\newtheorem{lem}[thm]{Lemma}
\newtheorem{cor}[thm]{Corollary}

\theoremstyle{definition}
\newtheorem{defn}[thm]{Definition}

\newtheorem{quest}[thm]{Question}

\DeclareMathOperator{\ad}{ad}

\newcommand{\field}[1]{\mathbb{#1}}

\newcommand{\R}{\field{R}}
\newcommand{\N}{\field{N}}
\newcommand{\K}{\field{K}}

\newcommand{\Quat}{\field{QU}}
\newcommand{\Oct}{\field{O}}
\newcommand{\C}{\field{C}}
\newcommand{\Hyp}{\field{H}}

\newcommand{\mcv}{\mathcal{V}}

\newcommand{\frn}{\mathfrak{n}}

\newcommand{\mfn}{\mathfrak{n}}

\newcommand{\mfg}{\mathfrak{g}}

\DeclareMathOperator{\Isom}{Isom}
\DeclareMathOperator{\Span}{Span}

\DeclareMathOperator{\Aut}{Aut}

\begin{document}

\title{Pinched Curvature in Heintze Groups of Carnot-type}

\author{Brendan Burns Healy}
\address{Department of Mathematical Sciences\\
        University of Wisconsin-Milwaukee\\
       Milwaukee, WI 53211\\
 USA}
\email{healyb@uwm.edu}

\begin{abstract}
Rank-one symmetric spaces carry a solvable group model which have a generalization to a larger class of Lie groups that are one-dimensional extensions of nilpotent groups. By examining some metric properties of these symmetric spaces, we motivate and prove the existence of analogous left-invariant, Riemannian metrics on Heintze groups of Carnot-type. These metrics adhere to certain natural curvature pinching properties, and we show in a special case that this pinching is optimal, appealing in part to a result of Belegradek and Kapovitch.
\end{abstract}




\maketitle

\section{Introduction}

Rank-one symmetric spaces of non-compact type are some of the most interesting and well-studied objects in mathematics. We begin by making an observation, well-known to experts, about the relationship between a Lie algebra structure these spaces admit and their sectional curvature as Riemannian manifolds, in order to motivate a generalization. Once a solvable group model is chosen for a negatively curved symmetric space, vertical tangent planes, which are defined as those that contain a vector orthogonal to the Heisenberg subalgebra, define subspaces isometric to real hyperbolic planes with constant curvature $-1$ inside $\K \Hyp^n$ when the other vector is in the direction of the first element of the lower central series of $\mathcal{H}^{n-1}$, and rescaled versions of real hyperbolic planes with curvature $-4$ when this vector is in the direction of the second element.

In \cite{Heintze74_Homogeneous} Heintze answers the question as to whether these symmetric spaces are the only examples of connected manifolds that are homogeneous and have negative curvature. He demonstrates that while they are not, all spaces of this form share some key features. In particular, all such Riemannian manifolds are isometric to Lie groups which are one-dimensional extensions of nilpotent groups (with other algebraic conditions) equipped with a particular kind of left-invariant metric. A Heintze group is a Lie group admitting a left-invariant Riemannian metric of negative curvature. The geometry of Heintze groups, has been an expanding area of study - see for example \cite{hei1}, \cite{hei2}, \cite{yves}, \cite{hei3}.

Heintze groups are said to be \emph{of Carnot-type} if their derived subgroup admits a stratification (see Section~\ref{sec:background}). We prove the existence of (left-invariant, Riemannian) metrics on Heintze groups of Carnot-type that generalize the metric-algebraic relationship of the solvable group model for rank one symmetric spaces. In the following statement $A$ refers to a vector orthogonal to the derived subalgebra, while a layered basis is an orthogonal basis which restricts to a basis on each element of the stratification.

\vspace{0.1in}

\noindent {\bf Theorem \ref{thm:main} } {\it
For any Heintze group of Carnot-type $G$ with stratification $\oplus_i \mcv_i$, there exists a left-invariant Riemannian metric such that at any point $p \in G$, if $X,Y$ are layered basis tangent vectors in $T_p G$, then the sectional curvature $K(X,Y)$ satisfies the following conditions:
\begin{itemize}
    \item  $K(A,Y) = -i^2$ for planes such that $Y \in \mcv_i$, and 
    
    \item $ -s^2 \leq K(X,Y) \leq -1 $ where $s$ is the step nilpotency of the derived subgroup.
\end{itemize} }

\vspace{0.1in}

One should consider this class of metric the higher-step analog of the metrics within the intersection $QP \cap AM$ considered in Section 6 of Eberlein and Heber \cite{QP}. 

These metrics for Heintze groups of Carnot-type speak to their role as a generalization of the solvable groups associated to rank-one symmetric spaces. In particular, they adhere, under a mild condition, to the same type of curvature pinching that negatively curved symmetric spaces do. In the following statement, we say that a Lie group has optimal pinching $\frac{1}{L}$ if it admits a left-invariant metric which realizes it as a Riemannian manifold with $\sec \in [-L,-1]$, but no left-invariant metrics with tighter curvature bounds.

\vspace{.1in}

\noindent {\bf Theorem \ref{thm:pinchedcurv} } {\it
Let $G$ be a Heintze group of Carnot-type whose derived subgroup admits a lattice. Then $G$ has optimal pinching $\frac{1}{s^2}$, where $s$ is the step nilpotency of $[G,G]$.}
\vspace{.1in}

To prove this statement, we construct a metric as in Theorem~\ref{thm:main} that adheres to certain additional parameters designed to ensure that the curvature of all other planes not mentioned in the definition of the metric lie in the right interval. In fact, this direction does not rely on the lattice assumption. The proof of the optimality of this pinching follows almost directly from the main theorem of Belegradek and Kapovitch in \cite{BeleKap}, in which they describe the pinching bounds on Riemannian manifolds in terms of the highest step nilpotent subgroup of the fundamental group.

Some bounds on the pinched curvature values for Heintze groups were already known. Work of Pansu in Section~5 of \cite{pansuConforme} can be adapted to demonstrate such bounds\footnote{That these bounds can be obtained from the work in \cite{pansuConforme} was communicated to the author by Gabriel Pallier.}. However, values obtained as a consequence of that work are less sharp than the curvature condition $-s^2 \leq K \leq -1$ we prove in Theorem~\ref{thm:pinchedcurv}. 

The proof of the Theorem~\ref{thm:main} proceeds by considering the `upper half-space model' of a Heintze group by way of the generalized Cayley transform (see \cite{harmonic}). This transformation is a diffeomorphism to the Heintze group expressed as a semidirect product from a manifold which is the derived subgroup cross the positive reals. This upper half-space thus inherits a class of pullback Riemannian metrics as well as a Lie structure. We then construct a metric on this half-space model with the desired properties, using the induced Lie bracket information from the diffeomorphism. Verification of the condition on the vertical tangent planes is done by way of the structure constants for a Lie group, which directly give an expression for sectional curvature; the horizontal (non-vertical) planes are examined using machinery developed in \cite{Heintze74_Homogeneous} and \cite{QP}.

The author would like to thank Mark Pengitore, Anton Lukyanenko, and Yves Cornulier for helpful conversations, and especially Gabriel Pallier for pointing out an error and helping with background, as well as the anonymous referee for constructive comments.

\section{Background}
\label{sec:background}

We begin with some definitions to understand the context that these groups arise from. A reader familiar with Heintze groups of Carnot-type may safely skip to the next section. As a convention, we will use capital letters to denote tangent vectors, or elements of the Lie algebra when appropriate. As usual, fraktur letters will denote the Lie (sub)algebras, and $g$ will refer to a Riemannian metric.

Recall that a metric space $X$ is \emph{homogeneous} if $\Isom(X)$ acts transitively on $X$. In particular, a Lie group equipped with a left-invariant metric is always homogeneous. A foundational result in the study of homogeneous manifolds is the following theorem of Kobayashi.

\begin{thm}{\cite{kobayashi}}
\label{thm:koba}
A homogeneous Riemannian manifold with nonpositive sectional curvature and negative definite Ricci tensor is simply connected.
\end{thm}

From here, significant attention was paid to the case of strictly negative sectional curvature.

\begin{defn}
A connected Riemannian manifold of strictly negative sectional curvature which is homogeneous is called a \emph{Heintze space}. 
\end{defn}

The name Heintze space derives from the following theorem of Heintze.

\begin{thm}{\cite{Heintze74_Homogeneous}}
\label{thm:Heintze}
Every Heintze space is isometric to a connected, solvable Lie group with a left-invariant, Riemannian metric. Furthermore, a Heintze space may be represented as a solvable Lie algebra with an inner product.
\end{thm}

Because of this identification, we will often move without comment between inner products on Lie algebras and left-invariant metrics on the associated manifold. In \cite{Heintze74_Homogeneous}, it is further observed that for such a solvable Lie group $\mathfrak{g}$, it is the case that $[\mathfrak{g},\mathfrak{g}]^\perp$ is one dimensional and that $ [\mathfrak{g},\mathfrak{g}]$ is nilpotent. A Lie group which admits a left-invariant metric realizing it as a Heintze space is called a \emph{Heintze group}. 

\begin{defn}
Let $G$ be a connected, simply connected, solvable Lie group with Lie algebgra $\mathfrak{g}$. The Lie algebra $[\mathfrak{g},\mathfrak{g}]$ is called the \emph{derived subalgebra} and the Lie group associated to $[\mathfrak{g},\mathfrak{g}]$ is called the \emph{derived subgroup}. A vector in $[\mathfrak{g},\mathfrak{g}]^\perp$ is called \emph{vertical}. A \emph{vertical plane} is a 2-dimensional subspace of $\mathfrak{g}$ that contains a vertical vector.
\end{defn}

Some important properties of these spaces come from Theorem~3 in \cite{Heintze74_Homogeneous}, which tells us that a solvable Lie algebra is one that comes from a Heintze space exactly when its derived subalgebra has codimension 1 and admits a contracting (equivalently expanding) automorphism. In certain cases of interest to us, this subalgebra will have stronger properties which allows us to pick a natural candidate for a metric on the associated manifold.

\begin{defn}
A \emph{stratification} on a nilpotent Lie algebra $\mathfrak{n}$ is a decomposition $ \mathfrak{n} = \mcv_1 \oplus \mcv_2 \oplus \ldots \oplus \mcv_s $ with the following properties:

\begin{itemize}
    \item $[\mcv_1, \mcv_j] = \mcv_{j+1}$
    \item $[\mcv_1, \mcv_s] = 0$
\end{itemize}

The subspaces $\mcv_i$ are called the \emph{layers} and $\mcv_1$ the \emph{horizontal layer}, and the simply connected, connected Lie group associated to $\mathfrak{n}$ is called a \emph{Carnot group}. More information on these groups is available in \cite{primer}.
\end{defn}

A stratification as above is also sometimes referred to as a \emph{Carnot grading}. Observe that a consequence of the above conditions is that $[\mcv_i, \mcv_j] \subset \mcv_{i+j}$. The focus on this specific subclass can be justified by the following result of Pansu.

\begin{thm}{\cite{pansu83}}
Let $G$ be a nilpotent Lie group equipped with a left-invariant Riemannian metric. Then the asymptotic cone of $G$ is isometric to a Carnot group equipped with a left-invariant sub-Riemannian metric.
\end{thm}

We should note from this setup that the degree of the stratification matches the step of nilpotency of the Lie group, recalling that $[\mathfrak{g},\mathfrak{g}]$ will always be nilpotent. However, not every nilpotent Lie group admits a stratification (see Examples 2.7 and 2.8 in \cite{primer}).

\begin{defn}
\label{defn:heintzecarnot}
A Heintze group $G = N \rtimes_\phi \R$ is said to be \emph{of Carnot-type} if the derived subalgebra admits a stratification and there is a nonzero value $\lambda$ such that semi-direct product is given by $\phi : \R \rightarrow \Aut(\mfn)$, where $\phi_t$ is defined by  $\phi_t (v_i) = e^{i \lambda t} v_i$ whenever $v_i \in \mcv_i$.
\end{defn}

Usually the normalizing assumption is made that $\lambda=1$. More background on Heintze groups of Carnot-type is available in \cite{newdirections} and \cite{yves}.


\section{Metrics on Heintze Groups of Carnot-type}
\label{sec:layered}

It is a well-known fact (a good exposition is available in \cite{anton}) that both $\R \Hyp ^n$ and $\C \Hyp ^n$ can be viewed as Heintze groups of Carnot-type equipped with a particular Riemannian metric. We use the properties that sectional curvature has in symmetric spaces to motivate the metric we will put on Heintze groups of Carnot-type. In particular we make the following observation, which is a classical fact about symmetric spaces.

Recall that a vertical vector is one such that $V \in \mathfrak{heis}^\perp$, the space of which must be one-dimensional by \cite{Heintze74_Homogeneous}. The following proposition can be found to be a consequence of Proposition~3.16 in \cite{QP}, proven for arbitrary Heintze groups of Carnot-type with derived subgroup of step 2.

\begin{prop}{\cite{QP}}
\label{prop:SymSpaceK}
Consider the Riemannian manifold $\K \Hyp^n$ for $\K \in \{ \R, \C, \Quat, \Oct\}$ viewed as an upper half-space with derived subalgebra $\mathfrak{heis} =\K^n \oplus \text{\emph{Im}} (\mathbb{K})$, normalized so that the supremum of its sectional curvature is $-1$. This space has the following curvature properties ($V,W \in T_p \K \Hyp^n$, $V$ vertical):
\[
K(V,W) = -1 \text{ if } W \in (dL_p)_e \K^n  \] \[
K(V,W) = -4 \text{ if } W \in (dL_p)_e \text{\emph{Im}} (\mathbb{K}) 
\]
\end{prop}

Recall that a metric on a Lie group $G$ is called \emph{left-invariant} if the left group action by $G$ is an action by isometries. We will only be interested in left-invariant, Riemannian metrics. In practice such a metric will come from a bilinear form on $\mathfrak{g}$, from which one defines the metric everywhere on $G$ by declaring the left group action to be by isometries. If $G$ is a Carnot group, we can make sense of the layers of $\mathfrak{g}$ at all points by declaring the following:

\begin{defn}
\label{def:layer}
Let $N$ be a stratified nilpotent Lie group with $\mathfrak{n} = \oplus_i^s \mcv_i$, where each $\mcv_i$ is equipped with a norm $|| \cdot ||_{\mcv_i}$. Define the \emph{layers} of $N$ to be the following sub-tangent bundles, defined at every point $p \in N$ as:
\[ \Delta^p_i := (dL_p)_e \mcv_i \]
where $dL_p$ represents the differential associated to left translation by $p$. Equip each of these subtangent spaces with the following norm:
\[  ||(dL_p)_e v || := ||v|| \]
\end{defn}

Our goal now is to pick out a Riemannian structure on an arbitrary Heintze group of Carnot-type that generalizes the properties we observe for rank-one symmetric spaces. Corollary~\ref{cor:nobiinv} below is intended to justify the use of a left-invariant, rather than bi-invariant, metric.

\begin{thm}{\cite{wolf}}
Let $N$ be a nilpotent Lie group which is not abelian. Then the Riemannian manifold obtained by equipping $\mathfrak{n}$ with an inner product which is extended to all tangent spaces by left-translation, satisfies the following condition at every point: \\

\noindent There exist tangent planes $R,S,T$ such that
\[ K(R) < 0 = K(S) < K(T) \]
where $K$ represents the sectional curvature.

\end{thm}

The following corollary follows from the observation that sectional curvature is always non-negative for Lie groups equipped with bi-invariant metrics (see, for instance, Chapter~4 of \cite{doCarmo_Riemannian}).

\begin{cor}{\cite{wolf}}
\label{cor:nobiinv}
Any nonabelian nilpotent Lie group does not admit a bi-invariant metric.
\end{cor}

Note that the information we have about arbitrary Heintze groups to this point is topological and algebraic in nature; we do not yet have a preferred way to equip the Lie algebra with an appropriate bilinear form that will allow us to measure distance and therefore curvature, such as is present in $\K \Hyp^n$. Instead of metrizing the semidirect product expression, we will use the following map to define an \emph{upper half-space model} for a Heintze group, which comes with a pullback Riemannian metric once a left-invariant metric is chosen on the derived subgroup.

\begin{defn}
{\cite{harmonic}}
\label{prop:harmonic}
Let $G$ be an Heintze group such that $G = N \rtimes \R$, where $N$ is the derived subgroup. Then the diffeomorphism \[ \Phi : N \times \R_+ \ni (n, y) \mapsto (n, \ln (y) ) \in G = N \rtimes \R \]
is called the \emph{generalized Cayley transform}. If $(G,g)$ is a Heintze group of Carnot-type equipped with a left-invariant Riemannian metric then
\[ \Phi^* g = \frac{1}{y^2} g|_{\mcv_1} + \frac{1}{y^4} g|_{\mcv_2} \ldots \frac{1}{y^{2s}} g|_{\mcv_s} + \frac{dy^2}{y^2} \]
is the \emph{Cayley pullback} of $g$, where $\sum g|_{\mcv_i}$ is a left-invariant metric on the derived subalgebra.
\end{defn}

We note that, importantly, this diffeomorphism preserves the set of vertical tangent vectors - the differential of this map preserves the property of being orthogonal to $N$ and the eigenspace decomposition of $\mathfrak{n}$ associated to the adjoint action of the vertical direction. In \cite{harmonic}, Nishikawa notes in Remark~3.1 that this choice of metric is far from unique, in both that the nilpotent metric is unspecified, and also that the $\ln$ function used could as easily be taken to be, for example, $\log_{e^2}$, which would result in denominators of $y, y^2 \ldots y^s, 4y^2$ in the formula for $\Phi^* g$. `In general', he remarks, `it is not clear a priori which choice should be canonical.'

We see that, when expressed as a matrix, the above tensor $\Phi^* g$ is block diagonal, where the size of each block is the rank of the vector space $\mcv_i$ with the final block being of size one. This block diagonal structure is true for any choice of isomorphism $\R \rightarrow \R_+$. Observe that the upper half-space model for rank one symmetric spaces represents a Cayley pullback metric.




\begin{defn}
Let $N$ be a Carnot group with a stratification of its algebra $\mathfrak{n} = \oplus_1^s \mcv_i$. A \emph{layered basis} on the tangent bundle of $N$ is the set left-translates (as in Definition~\ref{def:layer}) of orthogonal bases $\{e_{i,j} \}$ for each $\mcv_i$. In the event a metric is specified, we assume that a layered basis is orthonormal.
\end{defn}

Intersecting the layered basis with $T_p G$ will give an orthogonal basis for the nilpotent Lie algebra, and if the basis is taken together with $e_0$, a (vertical) vector in the positive $\R_+$ direction, we get a basis for the whole tangent space at an arbitrary point in a Heintze group of Carnot-type. We will abuse notation in this setting by referring to elements of any tangent space as linear combinations of $e_{i,j}, e_0$ without referencing the translation. We will also talk about a layered basis for a Heintze group of Carnot type as a layered basis for the Carnot subgroup together with a (unit) vertical vector. 

\begin{prop}{\cite{harmonic}}
\label{prop:adjoint}
Let $G$ be a Heintze group of Carnot-type. Then there exists a vertical vector $A$ such that when the Lie algebra decomposes $\mathfrak{g}= (\oplus \mcv_i) \oplus \R A$, then the $\mcv_i$ are eigenspaces of the action by $\ad A$ with eigenvalues $i$ in the Lie structure on $N \times \R_+$ induced by the generalized Cayley transform. 
\end{prop}

Furthermore, if we choose our Cayley pullback metric carefully, we may assume $[e_0,V_i] = i V_i$ (equivalently that $A$ is unit length). In essence, the generalized Cayley transform has replaced the adjoint action of $\R$, which is normally the exponential applied to a derivation, with the derivation itself.

We continue by providing an answer to the question of Nishikawa by proving the existence of the following class of Riemannian metric on any Heintze group of Carnot-type. 

\begin{thm}
\label{thm:main}
Let $G$ be a Heintze group of Carnot-type with a layered basis $\{ e_{i,j} \}$. Then $G$ admits a left-invariant Riemannian metric satisfying the following curvature conditions at any (equivalently every) point:

\[ -s^2 \leq K(e_{i,j},e_{k,\ell}) \leq -1 \]
\[ K(e_0, e_{i,j}) = -i^2 \]
where the set $\{e_{i,j}\}$ represents a layered basis and $e_0$ is vertical.
\end{thm}

A metric of this form will be called a \emph{layered metric}. Recall that vectors which are vertical in $N \rtimes \R$ remain vertical in the upper half-space model $N \times \R_+$.

Observe that layered metrics are `admissible' metrics in the sense of Definition~6.5 from Eberlein and Heber \cite{QP}. In particular, layered metrics on Heintze groups of Carnot-type based on 2-step groups will lie in the intersection of $QP \cap AM$, again in the sense of \cite{QP}.



Before presenting the proof of Theorem~\ref{thm:main}, we demonstrate the following proposition, which will get us half of the way there.

\begin{prop}
\label{prop:vert}
Let $G$ be a Heintze group of Carnot-type with Lie algebra $\mfg$ and $A$ be the vertical vector as in Proposition~\ref{prop:adjoint}. For any left-invariant metric $g$ on the derived subgroup, the Cayley pullback metric metric $h$ on $G$ which agrees with $g$ and sets $|A|_h=1$ has the following curvature property

\[ K_h( A, V_i ) = -i^2 \text{ for all $V_i \in \mcv_i$ }\]
\end{prop}

\begin{proof}
A result of Milnor from \cite{milnor} tells us that for such an orthonormal basis $\{ f_i\}$
\[ K(f_a, f_b) = \sum_k A_{ab}^k \]

where, if $\alpha_{ij}^k$ represents the structure constants for $G$, 
\begin{align*}
A_{ab}^k &= \frac{1}{2} \alpha_{ab}^k ( -\alpha_{ab}^k + \alpha_{bk}^a + \alpha_{ka}^b) \\
& \quad - \frac{1}{4}(\alpha_{ab}^k - \alpha_{bk}^a + \alpha_{ka}^b)(\alpha_{ab}^k + \alpha_{bk}^a - \alpha_{ka}^b) - \alpha_{ka}^a \alpha_{kb}^b . 
\end{align*}

Reindex the layered basis $\{ e_0, e_{i,j}\}$ to $f_k$ with a function $\psi(k) = i$ if $f_k = e_{i,j}$ for some $j$ ($f_k \in \mcv_{\psi(k)}$). We can observe the following identities of the structure constants:
\[ \alpha_{nn}^m = 0 \ , \ \alpha_{0m}^m = \psi(m)  \]
\[ \alpha_{nm}^\ell = 0 \text{ if } \psi(\ell) \neq \psi(n)+\psi(m) \]
\[ \alpha_{0m}^n = 0, n \neq m  .\]

To see why these relations hold, recall that structure constants satisfy
\[ \alpha_{nm}^\ell = \langle [e_n, e_m], e_\ell \rangle. \]

In this formulation, the first above property follows immediately from the skew-symmetry of the Lie bracket, and the second from the definition of a stratification and the fact that the layers are orthogonal. The last two properties follow from the fact that adjoint action on the layered basis by $\{e_0\}$ is exactly the adjoint action from Proposition~\ref{prop:adjoint}.

To prove the proposition, let $m$ be such that $e_{i,j} = f_m$. Note that $\alpha_{0m}^\ell = 0$ unless $\ell=m$, as the vertical direction is orthogonal to all layers, and each layer is orthogonal to others, but the adjoint action contains $f_m$ as an eigenvector with eigenvalue $\psi(m)$. Thus

\[ K(e_0, e_{i,j}) = K(e_0, f_m) =  A_{0m}^0 + A_{0m}^m .\]
\vspace{.05in}

From here we get that $A_{0m}^0 = 0$ when $f_m$ lies in the $i$-th layer $\mcv_i$, as it will only contain one nonzero term ($\alpha_{0m}^m$) which is multiplied by a term which is zero. Direct calculation shows that because $\alpha_{0m}^m = \psi(m)=i$, $A_{0m}^m = -i^2$, verifying that $K(e_0, e_{i,j})=-i^2$.
\end{proof}

Thus $|A|=1$ will be a condition on any metric we put on Heintze groups of Carnot-type. Furthermore, the above result remains true upon scaling the layered basis, as the zero terms must remain zero by bilinearity of the inner product and Lie bracket.

The paper \cite{Heintze74_Homogeneous} describes a process by which he guarantees a Lie group admits a metric of negative curvature that relies on a particular parameter which controls the extent to which the geometry of the space is dominated by the vertical direction. Because our goal will be to conclude pinched negative curvature (and not just negative curvature at all), we rely on more than just a sole parameter such as the one used in the statement of Theorem~2 of \cite{Heintze74_Homogeneous}. We name one of these parameters here which will be present in the proof of Theorem~\ref{thm:main}.

\begin{defn}
Let $(N,g)$ be a metrized Lie group. If $g$ satisfies the following condition for all $X,Y \in \mfn$
\[ || [X,Y]||_g \leq k \ ||X||_g \ ||Y||_g \ \]
we say that $g$ has \emph{Gromov value} $k$.
\end{defn}

Observe that all nilpotent groups admit left-invariant, Riemannian metrics with Gromov value $k$ for any $k>0$. 

\begin{thm}[Gromov, \cite{gromov_almostflat}]
Let $N$ be a nilpotent Lie group and $g$ a left-invariant metric on $N$. If $g$ has Gromov value $k$, then all sectional curvatures satisfy
\[ K^N \leq 100 k^2. \]
\end{thm}

\begin{proof}[Proof of Theorem~\ref{thm:main}]
Let $G$ be our Heintze group of Carnot-type and $\mathfrak{g}$ the associated Lie algebra. We will construct an inner product on $\mathfrak{g}$ which comes from a Cayley pullback metric $h$ that satisfies certain additional conditions. We begin by assuming that $|A|_h = 1$. We are now guaranteed the vertical plane condition by Proposition~\ref{prop:vert}. Thus we must only check the condition on basis planes contained in $\mathfrak{n}$. Specify a layered basis for $\mfn$ as in Definition~\ref{def:layer} and choose the $h$ makes that basis orthonormal (this does not affect the norm of $A$). We recall 1.5~b) from \cite{QP} the following curvature for such planes, assuming $X,Y$ are orthonormal. Notice the following is valid for any Riemannian metric on a Heintze group.
\[ K(X,Y) = K^\mathfrak{n}(X,Y) - |\phi(X,Y)|^2 \]

\vspace{.1in}

where $\phi(X,Y) = \sqrt{D_0}X \wedge \sqrt{D_0}Y$, and $D_0$ is the symmetric part of $\ad A$ - the skew-symmetric part is labelled $S_0$. We claim that $D_0$ has the same eigenvectors and eigenspaces as $\ad A$ in the setting of a Cayley pullback metric. Indeed this will be true if the transformation $\ad A \curvearrowright \mathfrak{n}$ is symmetric, as of course in this case $S_0=0$. We can see this is due to the assumption that the eigenvalues of this transformation are real - by assumption they are exactly the set $\{1, 2, \ldots s\}$ - and the eigenvectors are mutually orthogonal by the assumption that our metric respects the layered basis. Thus in fact, if the transformation is written in the basis of these elements, it is not just symmetric but even diagonal. We examine the value of $|\phi(V_i, V_j)|^2$ where $V_i, V_j$ are distinct orthonormal vectors in $\mcv_i, \mcv_j$ respectively ($i$ possibly equal to $j$).

\begin{align*}
    |\phi(V_i, V_j)|^2 &= |\sqrt{D_0}V_i \wedge \sqrt{D_0}V_j|^2 \\
    &= | \sqrt{i}V_i \wedge \sqrt{j}V_j |^2 \\
    &=|\sqrt{ij} (V_i \wedge V_j)|^2 \\
    &= (\sqrt{ij})^2 = ij \\
\end{align*}

Note that this value is always in the interval $[1, s^2]$. From here, because $h$ is a Cayley pullback metric, it agrees with some left-invariant Riemannian metric $g$ on $\mathfrak{n}$. 
We can now equip $\mfn$ with a metric that has Gromov value $\epsilon$ for our choice of $\epsilon >0$.
In particular, per 4.4 in \cite{gromov_almostflat}, this can be done simply by scaling the norm on certain vectors -- thus preserving orthogonality of vectors. One should think of this as first fixing lengths of the basis for $\mcv_1$, and then scaling down the basis of $\mcv_2$ as needed to ensure that
\[ ||[e_{1,i},e_{1,j}]|| \leq \epsilon ||e_{1,i}|| \ ||e_{1,j}|| \]
This is then done step by step for $\mcv_2, \mcv_3 \ldots \mcv_s$, with the assurance that this process terminates due to the fact that a) each $\mcv_i$ has finite rank and b) $\mcv_{\ell} = 0, \ell > s$. 
Thus we may assume that the restriction of $h$ to $\mfn$ induces a metric with Gromov value $\frac{1}{20s}$.
With this condition in place, we relabel our layered basis to one that is orthonormal, by choosing basis vectors of unit length and parallel to the original vectors. We will again label this basis $\{e_{i,j}\}$. Recall that this does not affect the curvature of vertical planes by Proposition~\ref{prop:adjoint} - $A$ remains unit length and the new metric is still a Cayley pullback metric. We compute the $N$--curvature for a plane contained in $\mcv_1$, recalling the structure constants and the result of Milnor:
\begin{align*}
 K^n(e_{1,a},e_{1,b}) &= \sum_k A_{ab}^k = \sum_{e_{2,k}} A_{ab}^k \\
 &= \sum_{e_{2,k}} \frac{1}{2} \alpha_{ab}^k ( -\alpha_{ab}^k + \alpha_{bk}^a + \alpha_{ka}^b) \\
& \quad - \frac{1}{4}(\alpha_{ab}^k - \alpha_{bk}^a + \alpha_{ka}^b)(\alpha_{ab}^k + \alpha_{bk}^a - \alpha_{ka}^b) - \alpha_{ka}^a \alpha_{kb}^b  \\
&= \sum_{e_{2,k}} \big( \frac{1}{2}\alpha_{ab}^k(-\alpha_{ab}^k) - \frac{1}{4}(\alpha_{ab}^k)^2 \big) \\
&= \sum_{e_{2,k}} -\frac{3}{4}(\alpha_{ab}^k)^2 = \sum_{e_{2,k}} -\frac{3}{4} \langle [e_{1,a},e_{1,b}],e_{2,k} \rangle^2 \\
&=-\frac{3}{4} ||[e_{1,a},e_{1,b}]||^2 \leq 0
\end{align*}

While of course not all planes will have nonpositive $N$--curvature (all left-invariant metrics on nilpotent Lie groups must have planes with both positive and negative curvature), we will see it is important that ones of this form will. In particular, we calculate the curvature of a plane of the form $\pi = \langle e_{i,j}, e_{k, \ell} \rangle$. First assume it is not the case that both $i=1, k=1$ or that both $i=s, k=s$. Then we see
\vspace{-.1in}
\begin{align*}
 K(e_{i,j},e_{k,\ell}) &= K^N(e_{i,j},e_{k,\ell}) - |\phi(e_{i,j},e_{k,\ell})|^2 \\
 &= K^N(e_{i,j},e_{k,\ell}) - ik 
 \end{align*}
 
Now by the assumption that $|K^N| < \frac{1}{4s^2} < \frac{1}{4}$ because we picked a metric with Gromov value $\frac{1}{20s}$, we are assured this value is in $[-1, -s^2]$. If $i=1, k=1$, then $K = K^N-1$ which is in $[-1,-s^2]$ because $ -\frac{1}{4s^2} < K^N<0$. For the case that $i=s, k=s$, we see that $e_{s,j},e_{s,\ell}$ are in the center of $\mfn$, and so also we have that $K^N(e_{s,j},e_{s,\ell}) =0$. This is also readily seen to be a consequence of the Milnor theorem and the fact that $[e_{s,j},e_{s,\ell}]=0$.

We now argue that all the conditions placed on the metric $h$ may be realized simultaneously. Observe that we insisted:

\begin{itemize}
    \item $h$ restricts to a metric on $\mathfrak{n}$ with Gromov value $\frac{1}{20s}$.
    \item $|A|_h=1$
    \item $h$ respects the orthogonal nature of a layered basis $\{e_{i,j}\}$ and is a Cayley pullback metric.
\end{itemize}

These conditions are compatible due to the orthogonal decomposition of the Lie algebra $\mathfrak{s} = \R A \oplus \frn = \R A \oplus \mcv_1 \oplus \ldots \oplus \mcv_s$, which allows us to scale $\R A $ independently from $\frn$, and $\mcv_i$ independently from $\mcv_j$. We also see that the assumption of a specific Gromov value does not change the orthogonality of the chosen basis, only scaling elements of the form $[e,e']$. Therefore we may scale this way for any metric induced on $\mathfrak{n}$ by $h$.
\end{proof}

\section{Pinched Curvature in Lie Groups}

We give definitions here of pinched curvature in Riemannian manifolds and optimal pinching in Lie groups. The goal of this section is an optimal pinching result for Heintze groups of Carnot-type whose derived subgroups admit lattices.

\begin{defn}
A Riemannian manifold with all sectional curvatures in the interval $[a,b]$ with $a,b< 0$ is said to be $\frac{b}{a}$\emph{--pinched}.
\end{defn}

Usually, the metric will be normalized to allow $b=-1$ and the pinching constant will be a fraction with numerator $1$; observe that scaling the metric to change the supremum of curvature leaves this fraction invariant. We see from the definition of pinched curvature that if a metric on a manifold is not strictly negatively curved or has the property that the infimum of its sectional curvature is $- \infty$, then it is not $C$--pinched for any $C$. Of course, the definition of pinching in a Lie group must be adjusted slightly, as different left-invariant metrics may give different values for sectional curvature.

\begin{defn}
A Lie group $G$ is said to have \emph{optimal pinching} $C$ if it admits a left-invariant metric which realizes it as a Riemannian manifold that is $C$ pinched, but no metrics which are $C'$--pinched for $C' > C$.
\end{defn}

One should observe that being $C'$ pinched for $C'>C$ is a stronger condition, because it means that the value of $b$ is closer to $-1$, and therefore the manifold is closer to being constant curvature, which is the strongest possible condition. Clearly some Lie groups are not exactly $K$--pinched for any value of $K$ - for instance the derived subgroup of a solvable Lie group may have codimension greater than 1.

Eberlein and Heber in \cite{QP} provide algebraic conditions on a solvable Lie algebra that guarantee it admits an inner product which turns the corresponding Lie group into a manifold with sectional curvature $ -(2^2) =-4 \leq K \leq -1 = -(1^2)$. Such a manifold is called quarter-pinched. We see layered metrics exhibit a very similar type of behavior, saving the fact that the ratio of maximum vertical curvature to minimum vertical curvature is now the reciprocal of an integral square.

\begin{thm}
\label{thm:pinchedcurv}
Heintze groups of Carnot-type have optimal pinching of $\frac{1}{s^2}$ when the derived subgroup admits a lattice, where $s$ is the step nilpotency of the derived subgroup.
\end{thm}

To prove this theorem, we first observe the following lemma.

\begin{lem}
\label{lem:pi}
Let $G$ be a Heintze group of Carnot-type equipped with a layered metric where the left-invariant metric on $N$ has Gromov value $\frac{1}{20s}$. Then for any tangent plane $\pi \subset \mathfrak{s}$ there exist orthonormal vectors $\alpha A + \beta U,V$ generating $\pi$ such that the following is true for the value $T=T(U,V)= \langle (\triangledown_U^N \ad A) V -(\triangledown_V^N \ad A) U, V \rangle$

\begin{itemize}
    \item $|T| < \frac{1}{2}$
    \item If $\pi \cap \mcv_1$ is non-trivial, then $T = 0$
    \item If $\pi \cap \mcv_s$ is non-trivial, then $T=0$
\end{itemize}
\end{lem}

\begin{proof}
We break down the definition of $T = \langle (\triangledown_U^N \ad A) V -(\triangledown_V^N \ad A) U, V \rangle$. Recall that for vectors regarded as left-invariant vector fields, we have the following identity (see (7) of \cite{azenwilson}): $\triangledown_X Y = \frac{1}{2} ( [X,Y]-(\ad_X)^T Y - (\ad_Y)^TX)$. In particular it is expressible as a linear combination of terms involving the adjoint action realized by the Lie bracket and its transpose. From this we see that a bound on the size of the Lie bracket given a bound on the inputs allows to bound the values of $\triangledown_U \ad A, \triangledown_V \ad A$ and thus $T$. \\
As a left-invariant vector field, we can realize $\ad A$ as the sum of coordinate vector fields using the diagonal expression from the layered basis: $\ad A = \sum_j \psi(j) E_j$ where $\psi(j) = i \iff E_j \in \mcv_i$. Then by the Gromov condition that $||[X,Y]|| \leq \frac{1}{20s}||X|| ||Y||$ and the facts that $\psi(j)\leq s$ and $||U|| \leq ||V|| = 1$, we know that $ ||(\triangledown_U^N \ad A) V -(\triangledown_V^N \ad A) U|| \leq \frac{3}{8} $, which means $|T| < \frac{1}{2}$, as we are projecting the previous vector onto the $V$ direction.

Now if $\pi \cap \mcv_1$ is nontrivial, we can arrange it so that $V \in \mcv_1$. Notice from the definition of $T$ that $(\triangledown_U^N \ad A) V -(\triangledown_V^N \ad A) U$ lives entirely in $\oplus_{i \geq 2} \mcv_i$, which is orthogonal to $\mcv_1$. As such $\langle (\triangledown_U^N \ad A) V -(\triangledown_V^N \ad A) U,V \rangle =0$.

Finally, if  $\pi \cap \mcv_s$ is nontrivial, we can use similar reasoning and assume that $U \in \mcv_s$. Now see that $\ad U$ is identically trivial, as $[U,V_j] \subset \mcv_{j+s} =0$. Similarily we know that $(\ad U)^T$ is also trivial, which proves the claim. \qedhere

\end{proof}

\begin{proof}[Proof of Theorem~\ref{thm:pinchedcurv}]
Let $G$ be a Heintze group of Carnot-type. Begin by equipping $G$ with a layered metric with the property that the left-invariant metric on the derived subgroup has a Gromov value of $\frac{1}{20s}$. By the definition of a layered metric, we already know the sectional curvature of planes determined by layered basis vectors, so we need to consider planes which may be generated by linear combinations of these. We break this argument into two cases for tangent planes $\pi$.

Case 1) $\pi \subset \mathfrak{n} = \R A  ^\perp$

Let $U,V$ be orthonormal vectors spanning $\pi$. Then we recall the formula from \cite{QP} used in the proof of Theorem~\ref{thm:main}.

\begin{align*}
 K(\pi) &= K^\mathfrak{n}(U,V) - |\phi(U,V)|^2 \\
 &= K^\mathfrak{n}(U,V) - |\sqrt{D_0}U \wedge \sqrt{D_0}V|^2  
\end{align*}

Recall that $\sqrt{D_0}V_i = \sqrt{i}V_i$ when $V_i \in \mcv_i$. Now because $|K^\mathfrak{n}| < \frac{1}{4s^2}$ the result follows from linearity of $D_0$ and the fact that $K^N$ is strictly negative when constrained to $\mcv_1$ and identically zero on $\mcv_s$.

Case 2) $\pi$ is not orthogonal to $A$

Assume $\pi = \langle \alpha A + \beta U, V \rangle, U,V \in \mathfrak{n}$ for vectors of the form specified in Lemma~\ref{lem:pi} - recall they are thus assumed to be orthonormal. In this case we use part c) of 1.5 in Eberlein and Heber \cite{QP} to find, recalling that $D_0 = \ad A$:

\begin{align*}
K(\pi) &= -\alpha^2 \langle N_0 V,V \rangle + \beta^2 K(U,V) +2 \alpha \beta \langle (\triangledown_U^N \ad A) V -(\triangledown_V^N \ad A) U, V \rangle  \\
&= -\alpha^2 \langle N_0 V,V \rangle + \beta^2 K(U,V) +2 \alpha \beta T(U,V)
\end{align*} 

In that statement, $N_0$ is defined to be the transformation defined in Heintze \cite{Heintze74_Homogeneous}; i.e. $N_0 = D_0^2 + [D_0,S_0]$. Thus, in our case, $N_0 = D_0^2$.

In the event that $\pi$ intersects either $\mcv_1$ or $\mcv_s$, the above simplifies by Lemma~\ref{lem:pi} to the following.

\[ K(\pi) = -\alpha^2 k^2 + \beta^2 K(U,V) \]

for $1 \leq k \leq s$. Because $K(U,V) \in [-s^2,-1]$, per case 1, and $\alpha^2 + \beta^2 =1$, this completes this case. In the event that $\pi$ does not intersect either the horizontal or the top layer nontrivially, our bounds for the above terms are instead $2 \leq k \leq s-1$ and $K(U,V) \in [-(s-1)^2,-4]$. Now we see that because Lemma~\ref{lem:pi} tells us that $|T|< \frac{1}{2}$, and so $|2 \alpha \beta T| < \frac{1}{2}$. Therefore $K \in [-(s-1)^2-0.5,-3.5] \subset [-s^2,-1]$.

To show this pinching is optimal, we will use the assumption that $[G,G]$ admits a lattice (this was not assumed in the ``forward'' direction). Begin by equipping $G$ with a metric, we wish to show it is not more tightly pinched that $\frac{1}{s^2}$. Immediately we see that if this metric does not have strictly negative curvature bounded away from zero, the result is immediate, as it will not be pinched for any value. Therefore we may normalize the metric to assume its largest curvature value is $-1$.

The derived subgroup is nilpotent; call this subgroup $N$. We know by Theorem~2.1 of \cite{Raghunathan} that a lattice in a nilpotent Lie group is always cocompact. Let $\Gamma$ represent the lattice assumed to exist in $N$. We can pass to the quotient manifold $G_\Gamma := G/\Gamma$. By construction $\pi_1(G_\Gamma) \cong \Gamma$. Furthermore, by cocompactness and Chapter~2 of \cite{Raghunathan}, we may conclude $\Gamma$ is also s-step nilpotent.

Observe that because the action $\Gamma \curvearrowright G$ is by isometries, the sectional curvatures present in $G$ is the same set of curvatures present in $G_\Gamma$; indeed just apply the locally isometric covering map. We may now apply the main result of Belegradek and Kapovitch from \cite{BeleKap}, which states that a manifold whose fundamental group has as a finite index subgroup a nilpotent group of step $s$ must have sectional curvatures, after normalization such that the supremum of its sectional curvature is $-1$, at least as negative as $-s^2$.
\end{proof}

In the above way we conclude that a layered metric observes the tightest possible constraints on sectional curvature when the Heintze group of Carnot-type admits a lattice of its Carnot group. 




\section{Further Discussion}
\label{sec:further}

Of primary interest when considering Heintze groups of Carnot-type has historically been solvable group models for symmetric spaces, or more generally Heintze groups whose associated nilpotent group is step 2. However, it is worthwhile to note examples of Heintze groups of Carnot-type which are of arbitrary step.

\begin{defn}{\cite{harmonic}}
Let $k \in \N$ and consider $SL(k+1,\R)$ the special linear group of degree $k+1$ with Lie algebra $\mathfrak{sl}(k+1,\R)$, where the bracket operation is matrix multiplication commutation. If $E_{ij}$ is the matrix with a $1$ in the $(i,j)$ place and zeroes elsewhere, let $$\mathfrak{n}= \Span \bigl( \{E_{ij} \ | \ j > i\} \bigr)$$ and let $A$ be the following matrix 

$$ A =	\frac{1}{2} \begin{pmatrix}
k & & & & \\
& k-2 & & & \\ 
& & \cdots & & \\ 
& & & -(k-2) & \\
& & & & -k
\end{pmatrix}	
$$

The Lie subalgebra of $\mathfrak{sl}(n+1,\R)$ formed by  $\mathfrak{n} \oplus \R A$ is denoted $\mathfrak{sc}(k)$. The \emph{special Heintze group of Carnot-type} is the corresponding Lie group is $SC(k)$.
\end{defn}

The condition that $j$ is strictly larger than $i$ guarantees all of the matrices $E_{ij}$ have trace zero and thus live in $\mathfrak{sl}(k+1)$.

The following is an easy application of the construction in Section~\ref{sec:layered}.

\begin{prop}
\label{thm:spech}
Consider the inner product $\langle \cdot   \rangle_H$ on $\mathfrak{sc}(k+1)$ such that:
\begin{itemize}
    \item $|E_{ij}|_H=\frac{1}{100k(j-i)}$ for all $i,j$
    \item $|A|_H=1$ 
    \item $\langle E_{ij},E_{k \ell} \rangle_H = \langle E_{ij},E_{k \ell} \rangle_{\mathfrak{sl}} = 0 $
    \item $\langle E_{ij},A \rangle_H = \langle E_{ij},A \rangle_{\mathfrak{sl}} = 0 $
\end{itemize}

Then $\langle \cdot  \rangle_H$ induces a layered metric on $SC(k)$.
\end{prop}

We list here a few questions motivated by the results of this article.

\begin{quest}
Is it true that the $\frac{1}{s^2}$--pinching constant is optimal for all Heintze groups of Carnot-type, regardless of the existence of lattices of the derived subgroup?
\end{quest}

The author suspects the answer to the above question is yes.

\begin{quest}
What does the space of layered metrics look like?
\end{quest}

This question arises from the obervation that the construction of layered metrics depends on a choice of Gromov value for the derived subgroup, as well as the normalization of a certain vertical vector.

\begin{quest}
What can be said about the Ricci and scalar curvatures of a layered metric?
\end{quest}

The fact that these Lie algebras are of Iwasawa type should help answer this question; see the formulas present in \cite{iwasawa-original}, or \cite{fili} for a more general context.

\bibliography{refs.bib}
\bibliographystyle{alpha}

\end{document}